\documentclass[a4paper]{amsart}

\usepackage{color}
\usepackage[dvipdfm]{hyperref}
\usepackage{amssymb}
\usepackage{amsthm}
\usepackage{amscd}
\usepackage{amsmath}
\usepackage[all]{xy}
\usepackage{graphicx}
\theoremstyle{plain}
\newtheorem{theorem}{Theorem}
\newtheorem*{theorem*}{Theorem}
\newtheorem*{Main}{Main Theorem}

\newtheorem*{maintheorem*}{Main Theorem}
\newtheorem{proposition}[theorem]{Proposition}
\newtheorem{corollary}[theorem]{Corollary}
\newtheorem{lemma}[theorem]{Lemma}
\newtheorem{claim}[theorem]{Claim}
\newtheorem*{conjecture*}{Conjecture}

\theoremstyle{definition}
\newtheorem{definition}{Definition}
\newtheorem*{definition*}{Definition}
\newtheorem*{example*}{Example}
\newtheorem*{notation*}{Notation}
\newtheorem*{notation-conv*}{Notation and convention}
\newtheorem*{convention*}{Convention}
\newtheorem*{hypothesis*}{Hypothesis}

\theoremstyle{remark}
\newtheorem{remark}{Remark}

\newcommand{\ie}{\emph{i.e.}\:}
\newcommand{\cf}{\emph{cf.}\:}

\newcommand{\mygeq}{\geqslant}

\newcommand{\ZZ}{{\mathbb Z}}
\newcommand{\CC}{{\mathbb C}}

\newcommand{\QQ}{{\mathbb Q}}
\newcommand{\IR}{{\mathbb R}}
\newcommand{\RR}{{\mathbb R}}
\newcommand{\NN}{{\mathbb N}}

\newcommand{\ii}{\mathbf{i}}
\newcommand{\jj}{\mathbf{j}}
\newcommand{\kk}{\mathbf{k}}

\newcommand{\SU}{{\mathrm{SU}(2)}}
\newcommand{\SL}{{\mathrm{SL}_2(\CC)}}

\newcommand{\rk}{\mathrm{rk}\,}

\newcommand{\Der}[2]{\mathrm{Der}_{#1} #2}
\newcommand{\Inn}[2]{\mathrm{Inn}_{#1} #2}

\newcommand{\Diff}[3]{\left. \frac{d}{d #2}{#1}\right\vert_{#3}}
\newcommand{\diff}[3]{\left. \frac{d #1}{d #2}\right\vert_{#3}}

\newcommand{\im}{\mathop{\mathrm{im}}\nolimits}

\newcommand{\su}{\mathfrak{su}(2)}
\newcommand{\sll}{\mathfrak{su}(2)}

\newcommand{\sllrho}{\mathfrak{su}(2)_{\rho}}

\newcommand{\Tor}[3]{\mathrm{Tor} ({#1}, {#2}, {#3})}

\newcommand{\bbrack}[1]{\lbrack \! \lbrack #1 \rbrack \! \rbrack}

\newcommand{\cbasis}[2][c]{\mathbf{#1}^{#2}}

\newcommand{\hbasis}[2][h]{\mathbf{#1}^{#2}}

\newcommand{\bord}{\partial}

\begin{document}

%
%


\title{
Behavior of the $\SU$-Reidemeister torsion form by mutation
}


\author{J\'er\^ome Dubois}

\address{Institut de Math\'ematiques de Jussieu -- Paris Rive Gauche, 
Universit\'e Paris Diderot--Paris 7, 
UFR de Ma\-th\'e\-ma\-ti\-ques,  
B\^atiment Sophie Germain, Case 7012
75205 Paris Cedex 13
France}
\email{dubois@math.jussieu.fr}
\date{\today}

\begin{abstract}
In this paper, we prove that the Reidemeister torsion twisted by the adjoint representation, which is considered as a 1-form, on the $\SU$-character variety of a knot exterior is invariant under mutation along a Conway sphere.
\end{abstract}

\maketitle

\tableofcontents



\section{Introduction}

If $K$ is a hyperbolic knot in $S^3$, we know that each mutant of $K$
is also hyperbolic and that their volumes are the same
(see~\cite{Ruberman}).  It is well--known that the Alexander
polynomial (\ie the abelian Reidemeister torsion) is invariant by
mutation like most of the classical or quantum knot invariants.  
In~\cite{Tillmann2000, Tillmann2004}, S. Tillmann have studied the behavior of the character variety of a knot group by mutation, and proved that generically the character varieties of a knot group and one of its mutant are birationally equivalent. 

In this paper, we study the behavior of the twisted Reidemeister torsion with coefficients in the adjoint representation associated to a generic $\SU$-representation, viewed as a $1$-form on the character variety, under a mutation. In fact, we prove that this kind of twisted Reidemeister torsion, with sign, is invariant
by positive mutation. Our technique is to use a ``cut and past argument'' which involves Mayer--Vietoris sequences and Turaev's refined version of torsion. To be more precise, consider a knot $K \subset S^3$, a positive mutation sphere $(F, \tau)$ and the associated mutant knot $K^\tau$ (see Section~\ref{S:mutation} for a complete definition). We let $M_K = S^3 \setminus N(K)$ denotes its exterior (here $N(K)$ is an open tubular  neighborhood of $K$) and $G_K = \pi_1(M_K)$ its group. We also consider the so-called regular part of the character variety $\mathrm{Reg}(K)$ which is the (open) set of all irreducible representation $\rho$ which satisfies $\dim H^1_\rho(M_K) = 1$ where $H^1_\rho(M_K)$ is the first cohomology group of $M_K$ with coefficients in the Lie algebra $\su_{Ad \circ \rho}$. Associated to any regular representation $\rho\colon G_K \to \SU$ is the torsion form $\tau^K_\rho$ which is a $1$-form on the character variety. One can prove that the mutation $\tau$ induces a diffeomorphism $\mathfrak{t}$ between (open) subsets of the regular parts of the character varieties of mutant knots (see Theorem~\ref{theorem:invariance_tors_mutation}).

\begin{Main}
Let $K, K^\tau, \tau, \mathfrak{t}$ as above. Suppose that $\tau$ is positive, then in a neighborhood of any regular representation $\rho\colon \pi_1(M_K) \to \SU$ whose restriction to $\pi_1(F)$ is irreducible, one has:
\[
\tau^K = \mathfrak{t}^* \: \tau^{K^\tau}.
\]
\end{Main}

In~\cite{MP}, P. Menal--Ferrer and J. Porti study the behavior by mutation of the Reidemeister torsion twisted by a representation into $\SL$ especially in the case of hyperbolic knots, and prove that it is invariant at the discrete and faithful representation corresponding to the complete structure. In~\cite{DFJ}, N. Dunfield, S. Friedl and N. Jackson make some computer computations to calculate the twisted Alexander invariant for some knots and their mutant. They observe that this invariant is not invariant by mutation. In~\cite{KL2}, P. Kirk and C. Livingston have already observed that some special twisted Alexander polynomials are not invariant by mutation and which could be used to distinguished some pairs of mutant knots.

\section*{Organization} 
The paper is organized as follows. Section~\ref{S:preliminaries} deals with some reviews on the needed tools: twisted cohomology, Reidemeister torsion (with sign) and Mayer--Vietoris property for Reidemeister torsion. In Section~\ref{S:torsionform}, we explain in details the construction of the torsion form, which is a 1-form on the character variety of the knot exterior. In Section~\ref{S:mutation} we describe the concept of mutation, give a precise definition of the notion of positive mutation, and explain how to associate to any $\SU$-representation of the group of a knot $K$ a representation of the group of one of its mutant. Section~\ref{S:proof} deals with the detailed proof of the Main Theorem (Theorem~\ref{theorem:invariance_tors_mutation}): the invariance of the torsion form by positive mutation. In last Section~\ref{S:open} we discuss some open problems related to mutation and Reidemeister torsions theory.

\section*{Acknowledgments}

The author would like to warmly thanks S. Friedl, J. Porti for helpful comments and encouragements related to the present work. He also would like to thanks the referee for his interesting remarks and corrections.


\section{Preliminaries}\label{S:preliminaries}

\subsection{Twisted cohomology and derivations}

In this subsection we review a method to describe the first twisted cohomology group by using twisted derivations.

	Let $G$ be a group of finite type, and consider a representation $\rho\colon G \to \SU$. The composition of a representation $\rho\colon G \to \SU$ with the adjoint action $Ad$ of $\SU$ on $\su$ 
gives us the following representation, called the adjoint representation associated to $\rho$:
\begin{align*}
  Ad \circ \rho \colon G &\to Aut(\su) = \mathrm{SO}(3) \\
 \gamma &\mapsto (v \mapsto \rho(\gamma) v \rho(\gamma)^{-1})
\end{align*}
A \emph{twisted derivation} (twisted by $Ad \circ \rho$) is a mapping $d\colon G \to \su$ satisfying the following cocycle relation:
	$$d(g_1g_2) = d(g_1) + Ad_{\rho(g_1)} d(g_2), \; \text{for all }g_1, g_2 \in G.$$
We let $\mathrm{Der}_\rho(G)$ denote the set of twisted derivation of $G$. 
Among twisted derivations, we distinguish the inner ones. A map $\delta\colon G \to \su$ is an \emph{inner derivation},  if there exists $a \in \su$ such that
	$$\delta(g) = a - Ad_{\rho(g)} a,\; \text{for all }g\in G.$$
We let $\mathrm{Inn}_\rho(G)$ denote the set of interior derivations of $G$ twisted by $\rho$. Observe that, if $\rho\colon G \to \SU$ is irreducible, then $\mathrm{Inn}_\rho(G) \cong \su$.

Let $W$ be a finite CW--complex and let $G = \pi_1(W)$ be its group. The Lie algebra $\su$ endows a structure of a (right) $\ZZ[G]$--module via the action $Ad \circ \rho$. In the sequel, $\sllrho$ denote this structure. Let $\widetilde{W}$ be the universal cover of $W$, it is well--known that the complex $C_*(\widetilde{W}; \ZZ)$ is also a (left) $\ZZ[G]$--module by using the action of $G = \pi_1(W)$ on $\widetilde{W}$ by the covering transformations. The {\it $\sllrho$-twisted cochain complex} of $W$ is
\[
C^*(W; \sll_\rho) = \mathrm{Hom}_{\ZZ[G]}\left(C_*(\widetilde{W}; \ZZ); \sll_\rho\right).
\] 
This twisted complex $C^*(W; \sll_\rho)$ computes the $(Ad \circ \rho)$-twisted cohomology 
$$H^*_{\rho}(W) = H^*\left(W; \sll_\rho\right),$$
which is a finite dimensional real vector space. 
It is well--know that (see~\cite{HS:1971}):
 $$Z^1_\rho(G) \cong \Der{\rho}{(G)}, \; B^1_\rho(G) \cong \Inn{\rho}{(G)}, \; H^1_\rho(G) \cong \Der{\rho}{(G)}/\Inn{\rho}{(G)},$$ and 
$$H^0_\rho(G) = \su^{Ad \circ \rho(G)} = \left\{v \in \su \; |\; v = Ad_{\rho(g)}v, \; \forall g \in G\right\}.$$ 
For each irreducible representation $\rho$ of $G$ we thus have the short exact sequence
\begin{equation}\label{ExSH}
\xymatrix@1@-.6pc{0 \ar[r] & \su \ar[r] & \Der{\rho}{(G)} \ar[r] & H^1_\rho(G) \ar[r] & 0.}
\end{equation}

\[
\mathrm{Der}_\rho(G) / \mathrm{Inn}_\rho(G) \cong H^1_\rho(G).
\]

If $X$ is a $K(G,1)$-space (for example knot exteriors are $K(\pi_1, 1)$-spaces), then 
$$\mathrm{Der}_\rho(X) = \mathrm{Der}_\rho(G), \; \mathrm{Inn}_\rho(X) = \mathrm{Inn}_\rho(G).$$

\subsection{Reidemeister torsion}

We review the basic notions and results about the sign--determined Reidemeister torsion introduced by Turaev which are needed in this paper. Details can be found in Milnor's survey~\cite{Milnor:1966} and in Turaev's monograph~\cite{Turaev:2002}.

\subsubsection*{Torsion of a chain complex}
Let $C_* = (\xymatrix@1@-.5pc{0 \ar[r] & C_n \ar[r]^-{d_n} & C_{n-1} \ar[r]^-{d_{n-1}} & \cdots \ar[r]^-{d_1} & C_0 \ar[r] & 0})$ be a chain complex of finite dimensional vector spaces over $\mathbb{R}$. Choose  a basis $\mathbf{c}^i$ for $C_i$ and  a basis $\mathbf{h}^i$ for the $i$-th homology group $H_i = H_i(C_*)$. The torsion of $C_*$ with respect to these choices of bases is defined as follows.

Let $\mathbf{b}^i$ be a sequence of vectors in $C_{i}$ such that $d_{i}(\mathbf{b}^i)$ is a basis of $B_{i-1}= \im(d_{i} \colon C_{i} \to C_{i-1})$ and let $\widetilde{\mathbf{h}}^i$ denote a lift of $\mathbf{h}^i$ in $Z_i = \ker(d_{i} \colon C_i \to C_{i-1})$. The set of vectors $d_{i+1}(\mathbf{b}^{i+1})\widetilde{\mathbf{h}}^i\mathbf{b}^i$ is a basis of $C_i$. Let $[d_{i+1}(\mathbf{b}^{i+1})\widetilde{\mathbf{h}}^i\mathbf{b}^i/\mathbf{c}^i] \in \mathbb{C}^*$ denote the determinant of the transition matrix between those bases (the entries of this matrix are coordinates of vectors in $d_{i+1}(\mathbf{b}^{i+1})\widetilde{\mathbf{h}}^i\mathbf{b}^i$ with respect to $\mathbf{c}^i$). The \emph{sign-determined Reidemeister torsion} of $C_*$ (with respect to the bases $\mathbf{c}^*$ and $\mathbf{h}^*$) is the following alternating product (see~\cite[Definition 3.1]{Turaev:2000}):
\begin{equation}
\label{Def:RTorsion}
\mathrm{Tor}(C_*, \mathbf{c}^*, \mathbf{h}^*) = (-1)^{|C_*|} \cdot  \prod_{i=0}^n [d_{i+1}(\mathbf{b}^{i+1})\widetilde{\mathbf{h}}^i\mathbf{b}^i/\mathbf{c}^i]^{(-1)^{i+1}} \in \mathbb{R}^*.
\end{equation}
Here  $$|C_*| = \sum_{k\geqslant 0} \alpha_k(C_*) \beta_k(C_*),$$ where $\alpha_i(C_*) = \sum_{k=0}^i \dim C_k$ and  $\beta_i(C_*)  = \sum_{k=0}^i \dim H_k$.

The torsion $\mathrm{Tor}(C_*, \mathbf{c}^*, \mathbf{h}^*)$ does not depend on the choices of $\mathbf{b}^i$ and $\widetilde{\mathbf{h}}^i$. 
Note that if $C_*$ is acyclic (i.e. if $H_i = 0$ for all $i$), then $|C_*| = 0$.


\subsubsection*{Torsion of a CW-complex}
	Let $W$ be a finite CW-complex; consider a representation $\rho\colon  \pi_1(W) \to \SU$. We let $\{e^{(i)}_1, \ldots, e^{(i)}_{n_i}\}$ denote the set of $i$-dimensional cells of $W$. Choose a lift $\tilde{e}^{(i)}_j$  of the cell $e^{(i)}_j$ in the universal cover $\widetilde{W}$ of $W$ and choose an arbitrary order and an arbitrary orientation for them. Thus, for each $i$, $\mathbf{c}^{i} = \{ \tilde{e}^{(i)}_1, \ldots, \tilde{e}^{(i)}_{n_i} \}$ is a $\ZZ[\pi_1(W)]$-basis of $C_i(\widetilde{W}; \ZZ)$ and we associate to it the corresponding ``dual" basis over $\mathbb{R}$ $$\mathbf{c}^{i}_{\su} = \left\{ {\tilde{e}^{(i)}_{1, \ii}, \tilde{e}^{(i)}_{1, \jj}, \tilde{e}^{(i)}_{1, \kk}, \ldots, \tilde{e}^{(i)}_{n_i, \ii}, \tilde{e}^{(i)}_{n_i, \jj}, \tilde{e}^{(i)}_{n_i, \kk}}\right\}$$ of $C^i(W; Ad \circ \rho) = \mathrm{Hom}_{\pi_1(X)}(C_i(\widetilde{W}; \ZZ), \su)$. Here 
\[
\ii = \left(\begin{array}{cc}i & 0 \\0 & -i\end{array}\right), \; \jj = \left(\begin{array}{cc} 0 & 1 \\1 & 0\end{array}\right) \text{ and }\kk = \left(\begin{array}{cc}0 & i \\ -i & 0\end{array}\right).
\]
	
If $\mathbf{h}^{i}$ is a basis of $H^i_\rho(W)$ then $\mathrm{Tor}(C^*(W; Ad\circ \rho), \mathbf{c}^*_{\su}, \mathbf{h}^{*}) \in \IR^*$ is well-defined.

	The cells $\{\tilde{e}^{(i)}_j\}_{0 \leqslant i \leqslant \dim W, 1 \leqslant j \leqslant n_i}$ are in one-to-one correspondence with the cells of $W$ and their order and orientation induce an order and an orientation for the cells $\{e^{(i)}_j\}_{0 \leqslant i \leqslant \dim W, 1 \leqslant j \leqslant n_i}$. We thus produce a basis over $\IR$ for $C_*(W; \IR)$ which is denoted $c^*$. 
	
	Choose a \emph{homology orientation} of $W$, \ie an orientation of the real vector space $H_*(W; \IR) = \bigoplus_{i\geqslant 0} H_i(W; \IR)$; let $\mathfrak{o}$ denote such an orientation. Provide each vector space $H_i(W; \IR)$ with a reference basis $h^i$ such that the basis $h^* = \{h^0, \ldots, h^{\dim W}\}$ of $H_*(W; \IR)$ is {positively oriented} with respect to the cohomology orientation $\mathfrak{o}$. Compute the sign-determined Reidemeister torsion $\mathrm{Tor}(C_*(W; \IR), c^*, h^{*}) \in \IR^*$  of the resulting based and cohomology based chain complex $C_*(W; \IR)$ and consider its sign $\tau_0 = \mathrm{sgn}\left(\mathrm{Tor}(C_*(W; \IR), c^*, h^{*})\right) \in \{\pm 1\}$. Further observe that $\tau_0$ does not depend on the choice of the positively oriented basis $h^*$. 
	
	The {sign-determined Reidemeister torsion} of the cohomology oriented CW-com\-plex $W$ twisted by the representation $Ad \circ \rho$ is the product
\[
\mathrm{Tor}(W; Ad\circ \rho, \mathbf{h}^{*}, \mathfrak{o}) = \tau_0 \cdot \mathrm{Tor}(C^*(W; Ad\circ \rho), \mathbf{c}^*_{\su}, \mathbf{h}^{*}) \in \IR^*.
\]
	The torsion $\mathrm{Tor}(W; Ad\circ \rho, \mathbf{h}^{*}, \mathfrak{o})$  is the \emph{$(Ad \circ \rho)$-twisted Reidemeister torsion} of $W$. It is well-defined. It does not depend on the choice of the lifts $\tilde{e}^{(i)}_j$ nor on the order and orientation of the cells (because they appear twice). Finally, it just depends on the conjugacy class of $\rho$.
	
	One can prove that $\mathrm{Tor}$ is invariant under cellular subdivision, homeomorphism class and simple homotopy type. In fact, it is precisely the sign $(-1)^{|C_*|}$ in~(\ref{Def:RTorsion}) which ensures all these properties of invariance (see Turaev's monographs~\cite{Turaev:2000, Turaev:2002}).

\subsection{Mayer--Vietoris sequence}

In this subsection we briefly review the so-called Ma\-yer-Vietoris formula for Reidemeister torsions which will be used in the proof of the main results. This formula is based on the multiplicativity property of Reidemeister torsions.

\begin{theorem}[Mayer-Vietoris formula]\label{FMV}
	Let $W$ be a finite CW-complex, let $W_1$ and $W_2$ be two closed subcomplexes such that $W=W_1 \cup W_2$ and $V=W_1 \cap W_2$ is not void. Consider any representation $\rho\colon \pi_1(W) \to \SU$ and let us $\rho_i = \rho_{|\pi_1(W_i)}$ and $\rho_V=\rho_{|\pi_1(V)}$ denote its restrictions respectively to $\pi_1(W_i)$ and $\pi_1(V)$. If  $\mathcal{H}$ denotes the Mayer--Vietoris sequence in twisted cohomology associated to the splitting  $W = W_1 \cup_V W_2$ and to the representation $\rho$, then one has the Mayer--Vietoris formula:
\[
\mathrm{Tor}(W_1; Ad \circ \rho_1) \cdot \mathrm{Tor}(W_2; Ad \circ \rho_2) = (-1)^{\varepsilon + \alpha} \mathrm{Tor}(W; Ad \circ \rho) \cdot \mathrm{Tor}(V; Ad \circ \rho_V) \cdot \mathrm{tor}(\mathcal{H}).
\] 
Here
\[
\alpha = \alpha(C^*(W; Ad \circ \rho), C^*(V; Ad\circ \rho_V))
\]
and
\[
\varepsilon = \varepsilon(C^*(W; Ad \circ \rho), C^*(W_1; Ad \circ \rho_1) \oplus C^*(W_2; Ad \circ \rho_2), C^*(V; Ad\circ \rho_V)).
\]
\end{theorem}
\begin{proof}
	This formula follows from the Multiplicativity Lemma for the torsions (see~\cite{Turaev:2002}) applied to the following sequence of complexes:
\[
\xymatrix@1@-.5pc{0 \ar[r] & C^*(W; Ad\circ \rho_V) \ar[r] & C^*(W_1; Ad \circ \rho_1) \oplus C^*(W_2; Ad \circ \rho_2) \ar[r] & C^*(V; Ad \circ \rho) \ar[r] & 0,}
\]
which induces the Mayer--Vietoris long exact sequence $\mathcal{H}$.
(see~\cite[Proposition 0.11]{Porti:1997}  for details).
\end{proof}

\section{Review on the construction of the torsion form}
\label{S:torsionform}

In this section, we are interested with knots. So, let $K$ be a knot in $S^3$ and consider its exterior $M_K = S^3 \setminus V(K)$, where $V(K)$ is an open tubular neighborhood  of $K$. Let $G_K = \pi_1(M_K)$ be the fundamental group of $M_K$, we call it the group of $K$. Observe that $M_K$ is a  compact connected three--dimensional manifold whose boundary consists in a single two--dimensional torus $\partial M_K = T^2$. 
 
It is well--known (see for example~\cite{Porti:1997} or~\cite{JDFourier}) that, for any representation $\rho \colon G_K \to \SU$, the $(Ad \circ \rho)$-twisted  cohomology never vanishes, actually one can prove using Poincar\'e duality that:
\[
\dim_{\RR} H^1_\rho(M_K) \geqslant 1.
\]
For an irreducible representation $\rho\colon G_K \to \SU$, we moreover know that $H^0_\rho(M_K) = 0$. As a consequence, using the fact that the Euler characteristic of $M_K$ vanishes, one has:
\[
\dim_{\RR} H^1_\rho(M_K) = \dim_{\RR} H^2_\rho(M_K),
\]
for any irreducible representation $\rho \colon G_K \to \SU$.

\subsection{The notion of regular representation}
In this subsection, we review the notion of regularity for representations of a knot group in $\SU$.

\begin{definition}
A representation $\rho \colon G_K \to \SU$ is called \emph{regular}, if $\dim_{\RR} H^1_\rho(M_K) = 1$.
\end{definition}

Of course, if a representation $\rho \colon G_K \to \SU$ is regular, then every conjugate of $\rho$ is also regular. Thus the notion of regularity is well--defined not only for representations but for characters. 
Important properties concerning the set $\mathrm{Reg}(K)$ of regular representations up to conjugation is summarized into the following result:

\begin{theorem}[see~\cite{Heu:2003, JDFourier}]\label{thm:reg}
{ The set $\mathrm{Reg}(K)$ is a one--dimensional manifold, and if $\rho \in \mathrm{Reg}(K)$ then the tangent space to the character variety $T_{\rho} X(M_K)$ is isomorphic to $H^1_\rho(M_K)$.} 
\end{theorem}

Moreover, for a regular representation $\rho \colon G_K \to \SU$, one has $\dim_{\RR} H^2_\rho(M_K) = 1$.

\subsection{Construction of the torsion form}
To construct a Reidemeister torsion, especially in a non acyclic context, we need to define some distinguished bases for homology groups. Let us review here Porti's construction of a distinguished basis for $H^2_\rho(M_K)$ (see~\cite{Porti:1997, JDFourier}).

To fix the notation, we suppose that $S^3$ is oriented and that we have chosen an orientation for the knot $K$. The knot exterior $M_K$ inherits an orientation from that of $S^3$ and its boundary $\partial M_K$ is also oriented using the convention ``the inward pointing normal vector in the last position". As $K$ is oriented, the peripheral system $(\mu, \lambda)$ inherits an orientation as follows. First, we orient $\mu$ by the rule $\ell k (\mu, K) = +1$, and $\lambda$ is oriented using the intersection number defined by the orientation of $\partial M_K$: $\mathrm{int}(\mu, \lambda) = +1$.

Let $\rho \colon G_K \to \SU$ be an irreducible representation. Observe that $\rho(\mu) \ne \pm \mathbf{1}$ (because, in a Wirtinger presentation of $G_K$ each generator is conjugate to $\mu$, and as $\rho$ is irreducible $\rho(G_K) \not\subset \{\pm  \mathbf{1}\}$), as a consequence there exist only one couple $(\theta, P_\rho) \in (0, \pi) \times S^2$ such that
\[
\rho(\mu) = \cos(\theta) + \sin(\theta) P_\rho.
\]
The vector $P_\rho$ generates the common axe of the rotations $Ad \circ \rho(\pi_1(\partial M_K))$, and thus $P_\rho$ generates $H^0_\rho(\partial M_K)$. 

Let us denote by $c$ the generator of $H^2(\partial M_K; \ZZ) = \mathrm{Hom}(H_2(\partial M_K; \ZZ), \ZZ)$ corresponding to the fundamental class $\lbrack \! \lbrack \partial M_K \rbrack \! \rbrack \in H_2(\partial M_K; \ZZ)$ induced by the orientation of $\partial M_K$. With such notation, one has (see~\cite{Porti:1997}, Proposition 1.3.2):
\begin{lemma}
The map $\phi^{(*)}_{P_\rho}\colon H^*_\rho(\partial M_K) \to H^*(\partial M_K; \RR)$ defined using the cup--product coupled with the killing form $\phi^{(*)}_{P_\rho}(z) = P_\rho  \smile z$ is an isomorphism.
\end{lemma}

The construction of the distinguished generator of $H^2_\rho(M_K)$ is based on the following result (see~\cite{Porti:1997}, Corollary 3.23):
\begin{lemma}
If $\rho\colon G_K \to \SU$ is regular, then the homomorphism $i^*\colon H^2_\rho(M_K) \to H^2_\rho(\partial M_K)$ induced by the inclusion $i \colon \partial M_K \hookrightarrow M_K$ is an isomorphism.
\end{lemma}

Combining these two lemmas, we construct an isomorphism:
\[
\phi^{(2)}_{P_\rho} \circ i^* \colon H^2_\rho(M_K) \to H^2(\partial M_K; \RR),
\]
and the distinguished generator $h^{(2)}_\rho \in H^2_\rho(M_K)$ is given by setting:
\[
h^{(2)}_\rho = {\left(\phi^{(2)}_{P_\rho} \circ i^*\right)}^{-1}(c).
\]

\begin{remark}
Observe that the distinguished generator $h^{(2)}_\rho$ does not depend on the orientation of $S^3$ but depends on the orientation of $K$.
\end{remark}

Now, to fix the ambiguity of the sign in the torsion, and following Turaev's construction of refined torsions~\cite{Turaev:2002}, we need to define an homological orientation. The knot exterior $M_K$ is equipped with a distinguished homology orientation denoted $\mathfrak{o}$ (see~\cite{Turaev:2002}) given by 
$$
 H_0(M_K;\IR) = \IR \bbrack{pt},\;
 H_1(M_K;\IR) = \IR \bbrack{\mu},\; 
 H_i(M_K;\IR) = 0 \, (i \mygeq 2).
$$
Here $\bbrack{pt}$ denotes the class of
a point and $\bbrack{\mu}$ denotes the class of the meridian $\mu$ of $K$. 

Let $\rho$ be a regular representation of $G_K$, in that case $T_{\rho} X(M_K)$ and $H^1_\rho(M_K)$ are isomorphic (see Theorem~\ref{thm:reg}), explicitly the isomorphism $\varphi_\rho\colon T_{\rho} X(M_K) \to H^1_\rho(M_K)$ is induced by (see~\cite[Paragraphe 3.1.3]{Porti:1997}) :
\begin{equation}\label{InclusionNaturelle}
\mathrm{T}_{\rho}{R(G_K)} \longrightarrow Z^1_\rho(G_K),\: {\left.\frac{d{\rho_t}}{dt}\right\vert}_{t=0} 
					\longmapsto \begin{cases}
   G_K \to \su   & \text{ } \\
   g \mapsto \left.\frac{d}{dt}{\rho_t(g)\rho(g^{-1})}\right\vert_{t=0}   & \text{}
\end{cases}
\end{equation}
with $\rho_0 = \rho$.

 The torsion form is the $1$-form defined by:
\[
\tau^K_\rho\colon T_{\rho} X(M_K) \to \CC, \quad \tau^K_\rho(v_\rho) = \begin{cases}
      \mathrm{Tor}(M_K; \mathfrak{su}(2)_\rho; \{\varphi_\rho(v_\rho), h^{(2)}_\rho\}; \mathfrak{o}) & \text{ if } v_\rho \ne 0 \\
      0 & \text{ if } v_\rho = 0
\end{cases}
\]
In this way, we have defined a volume form $\tau^K$ on the one--dimensional manifold $\mathrm{Reg}(K)$.
 
 Here is some remarks concerning the definition of the volume form $\tau^K$.
 
 \begin{remark}
 	For a regular representation $\rho\colon G_K \to \SU$, there exists a unique $P_\rho \in S^2$ and a unique $\bar{P}_\rho \in S^2$ such that:
$$\rho(\mu) = \cos(\theta) + \sin(\theta) P_\rho \text{ and } \rho(\mu) = \cos(2\pi-\theta) + \sin(2\pi-\theta) \bar{P}_\rho$$  with $\theta \in (0, \pi)$.
Changing $P_\rho$ into $\bar{P}_\rho$ in the construction has for consequence to change the volume form $\tau^K_\rho$ into $-\tau^K_\rho$.
 \end{remark}
 
 \begin{remark}
 The $1$-form $\tau^K$ does  not depend on the orientation of $K$.
  \end{remark}
  
   \begin{remark}
 In general $\mathrm{Reg}(K)$ is not compact, the problem of integration on the character variety and with respect to the torsion form $\tau^K$ is not easy. It has been considered recently in~\cite{FK:2011}.
  \end{remark}

\section{Mutation and mutant representations} 
\label{S:mutation}

\subsection{Review on mutation of knots}
Let $K \subset S^3$ be a knot. We let $F$ be a $4$-punctured
$2$-sphere which is incompressible in $M_K$ and whose closure in $S^3$
is a embedded sphere cutting $K$ transversally into four points. Such
a surface is called a \emph{mutation sphere}. We adopt the following
notation. The 3-sphere splits along $S^2$ into two 3-balls : 
$S^3 = B_1 \cup_{S^2} B_2$. Let $M_i = B_i \cap M_K$ and write 
$K_i = B_i \cap K$ for $i = 1, 2$. We have $K = K_1 \cup K_2$ and 
$M_K = M_1 \cup_{\mathrm{id}} M_2$ where $\mathrm{id} \colon F \to F$ is the
identity map. 
Note that each $K_i$ consists of two arcs.

The surface $F$ admits some orientation preserved involutions, we
consider the three $\pi$-angle rotations of $S^2$ that leave the four
points $K \cap S^2$ invariant. Let $\tau$ be such a rotation (see
Fig.~\ref{Fig:F}). The mutant knot $K^\tau$ is the knot $K_1 \cup_\tau K_2$ 
obtained by cutting $K$ along $F \cap K$ and gluing again after
the application of $\tau$. On Fig.~\ref{Fig:KT} one can find the
example of the Kinoshita-Terasaka knot $K_{KT}$ and its mutant the
Conway knot $K_C$.
 
The following diagram of natural inclusions is commutative:
\[
\xymatrix@R=1pt{&M_1 \ar[dr]^-{j_1} &\\
  F \ar[ur]^-{i_1}\ar[dr]_-{i_2}& & M_K\\
  & M_2 \ar[ur]_-{j_2} &}
\]
and the fundamental group of $F$, which is a free group of rank 3,
admits the following presentation:
\[
\pi_1(F) = \langle a,b,c,d \; |\; abcd = 1 \rangle \simeq
\mathbf{F}_3.
\]
 
Given a mutation sphere $(F, \tau)$, there is a fixed point of the
rotation $\tau$. In what follows, we chose this fixed point as the
base point of the fundamental groups: $\pi_1(F)$, $\pi_1(M_1)$,
$\pi_1(M_2)$ and $G_K = \pi_1(M_K)$.  Thus, using the Seifert--Van
Kampen Theorem, we get a decomposition for the group of $K$:
$$G_K \simeq \pi_1(M_1) \mathbf{*}_{\pi_1(F)} \pi_1(M_2).$$ 
Of course we get a similar decomposition for the group of the mutant
knot $K^\tau$:
$$G_{K^\tau} \simeq \pi_1(M_1) \mathbf{*}_{\tau_*(\pi_1(F))} \pi_1(M_2).$$ 
In particular, one can think of the representation space $R(G_K; \SU)$
as a subspace in $R(M_1; \SU) \times R(M_2; \SU)$ and the inclusion is
simply given by the restrictions to $\pi_1(M_1)$ and $\pi_1(M_2)$
(see~\cite{Tillmann2000}).
  
\begin{figure}[!hbt]
  \begin{center}
    \includegraphics[scale=1]{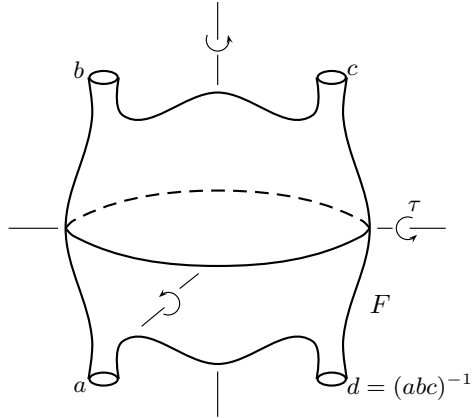}
    \caption{The punctured sphere $F$ and the rotations $\tau$.}
    \label{Fig:F}
  \end{center}
\end{figure}

\subsection{Positive and negative mutations}

Fix an orientation of the knot $K$. The orientation of $K$
induces an orientation for its two parts $K_i = B_i \cap K$ ($i=1,2$).  Moreover, the orientation of $K$ induces an orientation for its mutant $K^\tau = K_1 \cup_\tau K_2$ defined using the
orientation of the unchanged part $K_2$ of the knot. Each meridian $\mu$
of $K$ is oriented by the rule: $\ell k(\mu, K) = +1$.  All the curves
$\tilde{a}, \tilde{b}, \tilde{c}, \tilde{d}$ respectively
corresponding to the generators $a,b,c,d$ of $\pi_1(F)$ are oriented
using the same rule: $\ell k(\gamma, K) = +1$ where $\gamma
\in\{\tilde{a}, \tilde{b}, \tilde{c}, \tilde{d}\}$. Moreover observe
that necessarily the curves $\tilde{a}, \tilde{b}, \tilde{c},
\tilde{d}$ are coupled in two pairs where the curves in the same pair
belongs to the same component of $K_i$ ($i = 1, 2$).

We assign to each curve $\tilde{a},
\tilde{b}, \tilde{c}, \tilde{d}$ in $F$ a sign $\pm$ as follows.  Let
$\gamma \in\{\tilde{a}, \tilde{b}, \tilde{c}, \tilde{d}\}$, when
passing through $\gamma$ along the oriented knot $K$ if we go from
$M_1$ to $M_2$, then we assign $+$ to the curve $\gamma$, if not we
assign $-$. Of course this convention depends on the orientation of
the knot, if we reversed the orientation of $K$, then all signs
change. Moreover, two of the four curves are assign with $+$ and the
two other with $-$. Observe that if we consider a pair of curves which
lie on the same component of $K_i$, then necessarily one is assign
with $+$ and the other with $-$. A mutation $\tau$ sends the set of sign--oriented curves $\{\tilde{a}, \tilde{b}, \tilde{c}, \tilde{d}\}$ to itself (but the sign of the curves could be changed in the mutation). We say that the mutation $\tau$ is \emph{positive} if $\tau$ preserves signs, which means that for all $\gamma \in \{\tilde{a},\tilde{b}, \tilde{c}, \tilde{d}\}$, the curves $\gamma$ and $\tau(\gamma)$ in $F$ are assigned with the same sign. If not, we say that the mutation is \emph{negative}.

Observe that there exist only \emph{one}
positive mutation among the three possible ones and that this notion
does not depend on the orientation of $K$. We say that $K^\tau$ is a
\emph{positive mutant} (resp. \emph{negative mutant}) of $K$, if the mutation $\tau$ is positive (resp. negative). As
an example, the Conway knot $K_C$ is a positive mutant of the
Kinoshita--Terasaka knot $K_{KT}$ (see Fig.~\ref{Fig:KT}).

In what follows, we choose a common meridian for $K$ and $K^\tau$.
The meridian $\mu$ of $K$ is chosen to be a circle $\bord D^2 \times \{pt\}$ in 
$\bord M_2 = F \cup \bord D^2 \times I \cup \bord D^2 \times I$.  
The meridian of $K^\tau$ is chosen as same as $K$, and denoted by $\mu^\tau$.
Moreover we endow $M_K$ (resp. $M_{K^\tau}$) with the usual homological
orientations defined by the meridian $\mu$ (resp. $\mu^\tau$). 

\begin{figure}[!bht]
  \begin{center}
    \includegraphics[scale=1.1]{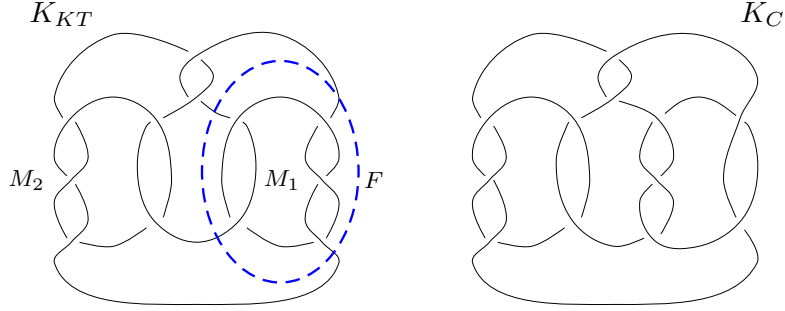}
    \caption{The Kinoshita-Terasaka knot $K_{KT}$ and its mutant the
      Conway knot $K_C$.}
    \label{Fig:KT}
  \end{center}
\end{figure}

\subsection{Some homology computations} 
\label{section:homology_computation}

The aim of this paragraph is to give some observations on
the (twisted and non--twisted) homology groups of the mutation sphere $F$ and on the manifolds
$M_i$ ($i = 1, 2$).

\begin{proposition}\label{prop:homology_F}
  The homology groups with real coefficients of the punctured
  $2$-sphere $F$ are described as follows:
  $$
  H_j(F;\IR) \simeq
  \begin{cases}
    \IR & \text{ if } j = 0, \\
    \IR^3 & \text{ if } j = 1, \\
    0 & \text{otherwise}.
  \end{cases}
  $$
\end{proposition}
\begin{proof}
  The mutation sphere $F$ is a $4$-punctured $2$-sphere, thus $F$ is
  homotopic to a bouquet of three circles, hence the homology groups
  $H_*(F;\IR)$ are the homology groups of a bouquet.
\end{proof}

\begin{remark}
  More precisely, one can observe that (see Fig.~\ref{Fig:F} for
  notation):
  \[
  H_1(F;\IR) \simeq \IR a \oplus \IR b \oplus \IR c \oplus \IR d /
  (a+b+c+d = 0).
  \]
\end{remark}

\begin{proposition}\label{prop:homology_M_i}
  The homology groups with real coefficients of the manifold $M_i$
  ($i=1,2$) are described as follows:
   $$
   H_j(M_i; \IR) \simeq
   \begin{cases}
     \IR  & j=0,\\
     \IR^2 & j = 1, \\
     0 & {\text otherwise.}
   \end{cases}
   $$
\end{proposition}
\begin{proof}
  It follows from the Mayer--Vietoris sequence for 
  the decomposition of 3-ball 
  $B^3_i = M_i \cup (D^2 \times I \cup D^2 \times I)$.
\end{proof}
\begin{remark}
  The knot $K$ is cut as four arcs and each $3$-ball $B_i$ contains
  two arcs denoted by $K_i$.  We can choose a pair of meridians for
  two arcs in $B_i$ as a basis of $H_1(M_i;\IR)$ and 
  denote by $\xi_i$ and $\eta_i$ the homology classes.
\end{remark}

 The Mayer--Vietoris sequence $\mathcal{V}_\IR$ with real coefficients associated
 to the splitting $M_K = M_{1} \cup_{\mathrm{id}} M_{2}$ is:
 \[
 \xymatrix@R=3pt{
   0 \ar[r] & 
     H_{1}(F;\IR) \ar[r]^-{(i^1_{*}, i^2_{*})} & 
     H_{1}(M_1;\IR) \oplus H_1(M_2;\IR) \ar[r]^-{j^1_* - j^2_*} & 
     H_1(M_K;\IR) & \\
   \ar[r]^-\delta & 
     H_0(F;\IR) \ar[r] & 
     H_{0}(M_1;\IR) \oplus H_0(M_2;\IR) \ar[r] & 
     H_0(M_K;\IR) \ar[r] & 
   0.}
\]

By counting dimensions, the connecting homomorphism $\delta$ is zero.
The first homology group $H_1(M_K; \IR)$ is generated by the meridian ${\mu}$.  Moreover, one has
the following properties on the connecting maps into $\mathcal{V}_\IR$.

\begin{lemma}\label{lemma:varphi_i}
  Let $\xi_i$ and $\eta_i$ denote meridians of $K_i$ such that
  they give a basis of $H_1(M_i;\IR)$ for $i=1, 2$.  
  In the Mayer--Vietoris
  sequence $\mathcal{V}_\IR$, the homomorphims $j^1_*$ and $j^2_*$ satisfy
  the following identities: 
  $$ j^1_*(\xi_1) = j^1_*(\eta_1) = 
   j^2_*(\xi_2) = j^2_*(\eta_2) = \bbrack{\mu} \text{ in } H_1(M_K;\IR),$$
   where $\mu$ denotes an oriented meridian of $K$.
\end{lemma}
\begin{proof}[Proof of the Lemma]
  The boundary of $M_i$ consists of a four punctured sphere $F$ and
  two annuli.  An annulus in $M_1$ is connected with two annuli in
  $M_2$ in $M_K$, thus joining these four annuli alternately, we obtain
  the boundary torus of $M_K$.  When an annulus in $\bord M_1$ has the
  boundary $a \cup (-b)$ on $F$ and contains the meridian $\xi_1$,
  the map $(i^1_*, i^2_*)$ sends $a$ and $b$ to $(\xi_1, \xi_2)$
  and $(\xi_1, \eta_2)$.  Such elements are contained in the
  kernel $j^1_* - j^2_*$.  Hence we have
  $j^1_*(\xi_1)=j^2_*(\xi_2) = j^2_*(\eta_2) = j^1_*(\eta_1)$.
  It follows from the surjectivity of $j^1_* - j^2_*$ that all
  $\xi_i$ and $\eta_i$ are send to the meridian $\mu$.
\end{proof}

\begin{remark}\label{remark:compare_homology_under_mutation}
  For a positive mutation $\tau$, Lemma~\ref{lemma:varphi_i}
  also holds for the Mayer-Vietoris sequence associated to the splitting $M_{K^\tau} = M_1 \cup_{\tau} M_2$.
  For a negative mutation $\tau$, it holds that $j^k_* (\xi_k) =
  j^k_*(\eta_k)$ for each $k=1, 2$.  But $j^1_*(\xi_1)$ has a
  different sign than the one of  $j^2_*(\xi_2)$.  
  \end{remark}

\subsection{Mutant representation}

For any representation $\rho \colon G_K \to \SU$, its restriction
$\rho_F \colon \pi_1(F) \to \SU$ to $\pi_1(F)$ is such that
$\chi_{\rho_F}(a) = \chi_{\rho_F}(b) = \chi_{\rho_F}(c) = \chi_{\rho_F}(d)$ (see Fig.~\ref{Fig:F}). 
We say that $\rho$ is \emph{$F$-irreducible} if its
restriction $\rho_F$ is irreducible. 

The following result computes the twisted homology groups of $F$ and $M_i$.
\begin{lemma}\label{lemma:twisted_homology_F_and_M_i}
If  $\rho\colon G_K \to \SU$ is an $F$-irreducible representation of $G_K$, then we have:
  \[
  \dim_{\CC} H^j_{\rho_F}(F) = 
  \begin{cases}
    6  & \text{ if } j = 1, \\
    0 & \text{otherwise}
  \end{cases}\] and
  \[
  \dim_{\CC} H^1_{\rho_i}(M_i) \geqslant 3.
  \]
\end{lemma}
\begin{proof}
\hfill\hphantom{}
  \begin{enumerate}
  \item The punctured sphere $F$ has the same homotopy type as a
    bouquet of $3$ circles, thus it has the same homotopy type as a
    $1$-dimensional CW-complex and its Euler characteristic is
    $-2$. Hence using the irreducibility of $\rho_F$, we conclude that all its
    twisted homology groups vanish except in degree $1$ for which:
    \[
    \dim_{\CC} H^1_{\rho_F}(F) = 
      - \chi(F) \cdot \dim_{\CC} \sll = 6.
    \]
  \item The boundary of the three--dimensional manifold $M_i$, $i=1,2$, is a surface of genus two, thus $\chi(M_i) = -1$. Moreover, as the representation $\varphi = {(\rho_i)}_{|\pi_1(F)}$ is irreducible, we observe that $\rho_i : \pi_1(M_i) \to \SU$ are also irreducible for $i=1,2$. The long exact sequence in twisted cohomology with coefficients in $Ad \circ \rho_i$ associated to the pair $(M_i, \bord M_i)$ reduces to:
 \[
 \xymatrix@-.5pc{0 \ar[r] & H^1_{\rho_i}(M_i, \bord M_i) \ar[r] & H^1_{\rho_i}(M_i) \ar[r]^-{i^*} & H^1_{\rho_i}(\bord M_i) \\
 &\textcolor{white}{H^1_{\rho_i}(M_i)}\ar[r]  &H^2_{\rho_i}(M_i, \bord M_i) \ar[r] & H^2_{\rho_i}(\bord M_i) \ar[r] & 0.}
 \]
 Observe that $\dim H^1_{\rho_i}(M_i) = 3 + \dim H^2_{\rho_i}(M_i)$ and using Poincar\'e duality, we obtain $\rk i^* = 3$, which implies that $\dim H^1_{\rho_i}(M_i) \geqslant 3$.

    \end{enumerate}
\end{proof}

Moreover, for restrictions to $\pi_1(F)$ of
$F$-irreducible representations we have the following lemma
(see~\cite[Theorem 2.2]{Ruberman} and~\cite[Lemma 2.1.1]{Tillmann2000}).

\begin{lemma}\label{lemma:tau_gievn_Ad}
  If $\psi \colon \pi_1(F) \to \SU$ is an irreducible representation
  such that $\chi_\psi(a) = \chi_\psi(b) = \chi_\psi(c) =  \chi_\psi(d)$,
  then there is an element $x \in \SU$ such that:
  \begin{equation}\label{eqn:psi}
    \psi \circ \tau_* = Ad_x \circ \psi.
  \end{equation}
\end{lemma}

\begin{remark}
  The element $x$ in Equation~(\ref{eqn:psi}) is not unique in general. 
  Actually, by Schur's lemma, $x$ is defined up to sign.
\end{remark}

The rest of this section consists in the construction of the so--called mutant representation associated to a representation of $G_K$. It is an $\SU$-representation of $G_{K^\tau}$
corresponding to a representation of $G_K$ obtained by twisting its restrictions to $R(M_1;\SU)$ and $R(M_2;\SU)$, 
using the pull--back of $\tau_*$.  Note that the pull--back
of $\tau_*$ is defined only on $R(F;\SU)$.  
However, Lemma~\ref{lemma:tau_gievn_Ad} says 
that the pull--back of $\tau_*$ is expressed as the adjoint action, so
we can use this adjoint action as the twisting on $R(M_1;\SU)$ instead
of the pull--back of $\tau_*$. 
 
Let $\rho \colon G_K \to \SU$ be an $F$-irreducible representation. 
The \emph{mutant representation} $\rho^\tau\colon G_{K^\tau} \to \SU$
associated to $x$ as in Equation~(\ref{eqn:psi}) is defined by
(see~\cite[Section 2.2]{Tillmann2000}):
\[
  \rho_1^\tau = \rho^\tau_{|\pi_1(M_1)} = 
  Ad_{x^{-1}} \circ \rho_1, \; 
  \rho_2^\tau = \rho_2,
\]
where $\rho_1 = \rho_{|\pi_1(M_1)}$ and $\rho_2 = \rho_{|\pi_1(M_2)}$.
One can observe that this definition is consistent because both parts
agree on the amalgamating subgroup. In that way, we have thus defined a map $\mathfrak{t} \colon \rho \mapsto \rho^\tau$ which only depends upon the inner automorphism defined by $x$.

It is easy to see that $\rho\colon G_K \to \SU$ is an irreducible
representation if and only if $\rho^\tau\colon G_{K^\tau} \to \SU$ is
as well (see for example~\cite{Ruberman} or~\cite{Tillmann2000}).

\begin{remark}
Let $\rho \colon G_K \to \SU$ be an $F$-irreducible representation. It is easy to observe that the restrictions $\rho_F = \rho_{|\pi_1(F)}$ and $\rho^\tau_F = \rho^\tau_{|\pi_1(F)}$ of respectively $\rho$ and $\rho^\tau$ to $\pi_1(F) $ coincide.
\end{remark}

Following ideas developed by D.~Cooper and D.~Long~\cite{CL:1996}, S.~Tillmann proved~\cite{Tillmann2000} that the geometric
components of the character varieties of $G_K$ and $G_{K^\tau}$ are
birationally equivalent. Here we adapt their arguments to the situation of $\SU$-representation spaces.

\begin{theorem}[\cite{CL:1996} Theorem~7.3 and \cite{Tillmann2000} Proposition~2.2.2]
Let $K$ be a knot in $S^3$ and consider a mutation sphere $(F, \tau)$. If $C$ denotes an irreducible component of the character variety $X(M_K)$ which contains at least one character of an $F$-irreducible representation, then $C$ is birationnaly equivalent to an irreducible component of $X(M_{K^\tau})$.
\end{theorem}

Let us review the main arguments of the proof.

\begin{proof}[Ideas of the proof]
 First observe that there is a one-one correspondence of characters of $F$-irreducible representations in $X(M_K)$ and $X(M_{K^\tau})$. We have to prove that the map $\mathfrak{t} \colon \rho \mapsto \rho^\tau$ is well defined for conjugate classes of $F$-irreducible representations. So let $\rho = (\rho_1,  \rho_2)$ and $\varrho = (\varrho_1, \varrho_2)$ be two conjugate $\SU$-representations of $\pi_1(M_K) \cong \pi_1(M_1) \star _{\pi_1(F)}\pi_1(M_2)$. There exist $z \in \SU$ such that $\varrho = Ad_{z} \circ \rho$. By construction of the associated mutant representations we have:
  \[
  \rho_1^\tau = Ad_{x^{-1}} \circ \rho_1 \text{ and } \varrho_1^\tau = Ad_{y^{-1}} \circ \varrho_1
  \]
  for some $x, y \in \SU$. Using the same linear algebra computation as in~\cite[Lemma 2.2.1]{Tillmann2000} one has:
\[
Ad_{z^{-1}y^{-1}} \circ {\varrho}_{|\pi_1(F)} = {\rho}_{|\pi_1(F)} = Ad_{x^{-1}z^{-1}} \circ {\varrho}_{|\pi_1(F)}.
\]
Since the restrictions of $\rho, \varrho$ to $\pi_1(F)$ are irreducible, we conclude that $zx = \pm yz$. As a result $\rho_1^\tau$ and $\varrho_1^\tau$ are conjugate by $z$. Hence $\varrho^\tau = Ad_z \circ \rho^\tau$ and $\mathfrak{t}$ is well defined on $F$-irreducible characters. Moreover, we can construct an inverse to that map because $\left( K^\tau\right)^\tau = K$.
  
  Let us denote $X_F(M_K)$ (resp. $X_F(M_{K^\tau})$) be the set of all conjugate classes of irreducible representations of $\pi_1(M_K)$ (resp. $\pi_1(M_{K^\tau}))$ whose restriction to $\pi_1(F)$ is also irreducible.  We have thus an isomorphism $\mathfrak{t}$ between $X_F(M_K)$ and $X_F(M_{K^\tau})$.
  
  The birationnality of $\mathfrak{t}$ comes from the same arguments used in~\cite[\S~2.2]{Tillmann2000}.
\end{proof}

We close this paragraph on the mutant representation by a remark.

\begin{remark}[A digression on $\SL$-character variety and hyperbolic knots]
If $K$ is a hyperbolic knot, then there is a discrete and faithful
representation of $G_K$ into $\mathrm{PSL}_2(\CC)$ which lifts to a
representation $\rho_0\colon G_K \to \SL$. Such a representation is
irreducible. D.~Ruberman proved~\cite{Ruberman} the following result
about the discrete and faithful representation (see
also~\cite[Corollaries 3 and 4]{Tillmann2004}).

\begin{proposition}
  Let $K$ be a hyperbolic knot and consider a mutation sphere $(F,
  \tau)$. Then the mutant knot $K^\tau$ is hyperbolic (with the same
  volume as $K$), and the discrete and faithful representation
  $\rho_0\colon G_K \to \SL$ is $F$-irreducible, moreover the
  corresponding mutant representation $\rho_0^\tau\colon G_{K^\tau} \to \SL$ 
  is the discrete and faithful representation of the
  hyperbolic structure of $K^\tau$.
\end{proposition}
\begin{proof}[Sketch of the proof]
  The part on hyperbolicity and volume is~\cite[Corollary 1.4]{Ruberman}.

  Let $\psi_0$ be the restriction of $\rho_0$ to $\pi_1(F)$. One can
  observe that as a restriction, $\psi_0$ is discrete and faithful, so
  it is in particular irreducible thus $\rho_0$ is $F$-irreducible.

\end{proof}

\end{remark}


\subsection{Regularity property of mutant representations}

In the previous subsection, we construct a map $\mathfrak{t}\colon X(M_K) \to X(M_{K^\tau})$ which is an isomorphism from $X_F(M_K)$ into $X_F(M_{K^\tau})$. Here we are interested in the corresponding tangent map at an $F$-irreducible representation $\rho\colon G_K \to \SU$.  Recall that for a regular representation one has $T_\rho X(M_K) \cong H^1_\rho(M_K)$. To understand the tangent map corresponding to $\mathfrak{t}\colon X(M_K) \to X(M_{K^\tau})$ at a regular and $F$-irreducible representation $\rho$ we interpret it at the level of $(Ad \circ \rho)$-twisted cohomology group and explicitly construct an isomorphism 
\begin{equation}\label{isom_H1}
\tau^{\sharp} : H^1_\rho(M_K) \to H^1_{\rho^\tau}(M_{K^\tau}).
\end{equation}
 Again, $\rho_i = \rho_{|\pi_1(M_i)}$, $i = 1, 2$, and $\rho_F = \rho_{|\pi_1(F)}$	 denotes the restrictions of $\rho$.

\begin{notation*}
In the sequel, we use the following notation:
\[
{}^{x}\!\rho = Ad_x \circ \rho, \text{ for } x \in \SU \text{ and } \rho \text{ a representation}.
\]
\end{notation*}

Let $\rho$ be any representation of $G_K$. The construction of the isomorphism of Equation~(\ref{isom_H1}) is based on the following technical result.
\begin{lemma}
	If $z \in \Der{\rho_F}{(F)}$, then $z \circ \tau_* \in \Der{^{x}\!\rho_F}{(F)}$ is such that: 
\begin{equation}
\label{E:Der}
	z \circ \tau_* = {}^{x}\!z + \delta
\end{equation}
with some $\delta \in \Inn{^{x}\!\rho_F}{(F)}$, where ${}^{x}\!z = Ad_x \circ z$.
\end{lemma}
\begin{proof}
The proof of Equation~(\ref{E:Der}) essentially consists in writing down the derivative of the equality $\rho_F \circ \tau_* = {}^{x}\!\rho_F = Ad_x \circ \rho_F$ (see Lemma~\ref{lemma:tau_gievn_Ad}).

It is easy to observe that $z \circ \tau_* \in \Der{^{x}\!\rho_F}{(F)}$. Let $\varphi_t : \pi_1(F) \to \SU$ be a germ at origin such that $\varphi_0 = \rho_F$ and satisfying, for all $g \in \pi_1(F)$, the following identity:
	$$z(g) = \Diff{\varphi_t(g) \rho_F(g)^{-1}}{t}{t=0}.$$
For all $t$ in a neighborhood of $0$, there exit $x_t \in \SU$ such that $\varphi_t \circ \tau_* = {}^{x_t}\!\varphi_t = Ad_{x_t} \circ \varphi_t$. 
Set $X=\diff{x_t}{t}{t=0}$ and take the derivative of the preceding equality with respect to $t$, one has, for all $g \in \pi_1(F)$,
 $$\diff{\varphi_t(\tau_*(g))}{t}{t=0} = X \rho_F(g) x^{-1} + a \diff{\varphi_t(g)}{t}{t=0}x^{-1} - x \rho_F(g)x^{-1}Xx^{-1}.$$
 Moreover, $\rho_F(\tau_*(g))=x \rho_F(g) x^{-1}$, thus
 \begin{multline}
\diff{\varphi_t(\tau_*(g))}{t}{t=0} {\rho_F(\tau_*(g))}^{-1}= Xx^{-1} - x \rho_F(g)x^{-1}(Xx^{-1})x \rho_F(g)x^{-1} \\ + x \diff{\varphi_t(g)}{t}{t=0} {\rho_F(g)}^{-1} x^{-1}.\notag
\end{multline}
Finally, for all $g \in \pi_1(F)$, 
$$z \circ \tau_*(g) = {}^{x}\!z(g) + (1-Ad_{{}^{x}\!\rho_F(g)}) Xx^{-1},$$
with $Xx^{-1} \in \sll$, 
\end{proof}

From this technical lemma, we deduce:
\begin{corollary}\label{Cor:Conj}
	If $h \in H^1_{\rho_F}(F)$, then $h \circ \tau_* \in H^1_{{}^{x}\!\rho_F}(F)$ and $h \circ \tau_* = {}^{x}\!h = Ad_x \circ h$.
\end{corollary}

Observe that the twisted cohomology groups $H^1_{\rho_1}(M_1)$ and $H^1_{{}^{x^{-1}}\!\rho_1}(M_1)$ are isomorphic. Moreover,  the isomorphism is induced by $\phi_x\colon z \mapsto Ad_{x^{-1}}z$ and will be denoted $\bar{\phi_x}$ in the sequel. Using Corollary~\ref{Cor:Conj} it is easy to deduce the following result.

\begin{claim}\label{claim:commutativity}
	Let $i_\ell\colon F \hookrightarrow M_\ell$, $\ell=1,2$, be the usual inclusions. The following diagram is commutative:
\begin{equation}\label{Diag}
\xymatrix{
H^1_{\rho_1}(M_1) \oplus H^1_{\rho_2}(M_2) \ar[d]_-{\bar{\phi_x} \oplus \mathrm{Id}}^-\cong \ar[r]^-{i_1^* + i_2^*} & H^1_{\rho_F}(F) \\
H^1_{{}^{x^{-1}}\!\rho_1}(M_1) \oplus H^1_{\rho_2}(M_2) \ar[ur]_-{\tau^*i_1^* + i_2^*} &
}
\end{equation}
\end{claim}

Let $j_\ell\colon M_\ell \hookrightarrow M_K$ and $j'_\ell\colon M_\ell \hookrightarrow M_{K^\tau}$, $\ell=1, 2$, be the usual inclusions. Write down the Mayer-Vietoris sequences in cohomology respectively associated to the splittings $M_K = M_1 \cup_{\mathrm{Id}} M_2$ and $M_{K^\tau} = M_1 \cup_{\tau} M_2$ and twisted by $\rho$ and $\rho^\tau$. We obtain:
\begin{equation}
\label{Suite1}
\xymatrix@1@+.5pc{0 \ar[r] & H^1_\rho(M_K) \ar[r]^-{j_1^* \oplus -j_2^*} & H^1_{\rho_1}(M_1) \oplus H^1_{\rho_2}(M_2) \ar[r]^-{{i_1^*+i_2^*}} & H^1_{\rho_F}(F) \ar[r] &  \cdots}
\end{equation}
and
\begin{equation}
\label{Suite2}
\xymatrix@1@+.9pc{0 \ar[r] & H^1_{\rho^\tau}(M_{K^\tau}) \ar[r]^-{{j'_1}^* \oplus -{j'_2}^*} & H^1_{\rho_1^\tau}(M_1) \oplus H^1_{\rho_2^\tau}(M_2) \ar[r]^-{{\tau^*i_1^*+i_2^*}} & H^1_{\rho_F^\tau}(F) \ar[r] &  \cdots}
\end{equation}
where $\rho_1^\tau = {}^{x^{-1}}\!\rho_1$, $\rho_2^\tau=\rho_2$ and $\rho_F^\tau = \rho_F$.

Combine these two exact sequences by using Diagram~(\ref{Diag}), one has the following commutative diagram:
\[
\xymatrix@-.7pc{0 \ar[r] & H^1_\rho(M_K) \ar[r] & H^1_{\rho_1}(M_1) \oplus H^1_{\rho_2}(M_2) \ar[r] \ar[d]_-{\bar{\phi_x} \oplus \mathrm{Id}}^-\cong& H^1_{\rho_F}(F) \ar[r] \ar[d]^-=& H^2_\rho(M_K) \ar[r]&  \cdots \\
0 \ar[r] & H^1_{\rho^\tau}(M_{K^\tau}) \ar[r] & H^1_{\rho^\tau_1}(M_1) \oplus H^1_{\rho^\tau_2}(M_2) \ar[r] & H^1_{\rho_F}(F) \ar[r] & H^2_{\rho^\tau}(M_{K^\tau}) \ar[r]&\cdots}
\]
Thus we can restrict the isomorphism 
$$\bar{\phi_x} \oplus \mathrm{Id}\colon H^1_{\rho_1}(M_1) \oplus H^1_{\rho_2}(M_2) \to H^1_{\rho^\tau_1}(M_1) \oplus H^1_{\rho^\tau_2}(M_2)$$ to an isomorphism 
\begin{equation}\label{isomutant}
\tau^\sharp\colon H^1_\rho(M_K) \to H^1_{\rho^\tau}(M_{K^\tau}).
\end{equation}

An immediate consequence is the following: 

\begin{theorem}
A representation $\rho\colon G_K \to \SU$ is regular if, and only if, its mutant representation $\rho^\tau \colon G_{K^\tau} \to \SU$ is also regular.
\end{theorem}

To compare the torsion form of $K$ and the one of one of its mutant $K^\tau$ we make the following technical hypothesis:

\begin{hypothesis*}
  Fix a (positive) mutation $(F, \tau)$ and suppose that
  $\rho\colon G_K \to \SU$ is \emph{regular} and \emph{$F$-irreducible}, which means that:
  \begin{enumerate}
  \item $\rho$ is regular  (\ie $\rho$ is irreducible and $\dim H^1_\rho(M_K) = 1$); 
  \item and the restriction $\rho_F$ of $\rho$ to $\pi_1(F)$ is irreducible.
\end{enumerate}
  \end{hypothesis*}
  
  Here is some other formulations of the preceding hypothesis.
  
  \begin{claim}
  Let  $\rho\colon G_K \to \SU$ be a regular representation. If $\rho$ is $F$-irreducible then $H^2_{\rho_i}(M_i) = 0$, for $i=1,2$.
  \end{claim}
  
  \begin{proof}
  Consider a regular representation $\rho\colon G_K \to \SU$ which is also $F$-irreducible. One has $H^0_{\rho_F}(F) = H^2_{\rho_F}(F) = 0$ (see Lemma~\ref{lemma:twisted_homology_F_and_M_i}) and write down the Mayer-Vietoris sequence associated to the splitting $M_K = M_1 \cup_{\mathrm{Id}} M_2$:
  \begin{equation}\label{MV1}
  \xymatrix@1@-.3pc{\cdots \ar[r] & H^1_{\rho_1}(M_1) \oplus H^1_{\rho_2}(M_2) \ar[r] & H^1_{\rho_F}(F) \ar[r]^-\delta & H^2_\rho(M_K) \ar[r] & H^2_{\rho_1}(M_1) \oplus H^2_{\rho_2}(M_2) \ar[r] & 0.}
  \end{equation}
  We prove that  the connecting homomorphism $\delta\colon H^1_{\rho_F}(F) \to H^2_\rho(M_K)$ is onto as follows. One can view this homomorphism as the composition of the three following homomorphisms:
  \begin{itemize}
  \item $H^1_{\rho_F}(F) \to H^1_{\rho_F}(\partial F)$ induced by the the usual inclusion $\partial F \hookrightarrow F$,
  \item $H^1_{\rho_F}(\partial F) \to H^2_{\rho}(\partial M_K)$ which is the restriction to the boundary of the connecting homomorphism $\delta\colon H^1_{\rho_F}(F) \to H^2_\rho(M_K)$ appearing in the Mayer-Vietoris sequence,
  \item $H^2_{\rho}(\partial M_K) \to H^2_\rho(M_K)$ is the inverse of the isomorphism $H^2_\rho(M_K) \to H^2_{\rho}(\partial M_K)$ (as $\rho$ is regular) induced by the usual inclusion $\partial M_K \hookrightarrow M_K$.
\end{itemize}
All of these three homomorphisms are onto and thus $\delta\colon H^1_{\rho_F}(F) \to H^2_\rho(M_K)$ is also onto. From the Mayer-Vietoris sequence in Equation~(\ref{MV1}), we conclude that $H^2_{\rho_i}(M_i) = 0$, for $i=1,2$.
  \end{proof}
    
  \begin{remark}
  It is easy to prove, using the proof of Lemma~\ref{lemma:twisted_homology_F_and_M_i}, that for a regular and $F$-irreducible representation $\rho\colon G_K \to \SU$, one has:
  \begin{enumerate}
  \item $\dim H^1_{\rho_i}(M_i) = 3$, for $i=1, 2$,
  \item the homomorphism $H^1_{\rho_i}(M_i, \bord M_i) \to H^1_{\rho_i}(M_i)$ is 0, for $i=1, 2$.
\end{enumerate}
  \end{remark}
  
  \begin{remark}
  If  $\rho\colon G_K \to \SU$ is regular and $F$-irreducible, then, all its restrictions $\rho_F = \rho_{|\pi_1(F)}$ and $\rho_i = \rho_{|\pi_1(M_i)}$ ($i = 1, 2$) are {irreducible}.
\end{remark}

\section{Behavior of the torsion form by positive mutation}
\label{S:proof}

In this section, we prove the Main Theorem which asserts that the torsion form is invariant by \emph{positive} mutation using the notation introduced in the previous section.

\begin{theorem}\label{theorem:invariance_tors_mutation}
  If $\rho\colon G_{K} \to \SU$ is a regular and $F$-irreducible
  representation, then
  \[
  \tau^{K^\tau}_{\rho^\tau} \circ \tau^{\sharp} = \tau^{K}_{\rho}.
  \]
  Here $\rho^\tau\colon G_{K^\tau} \to \SU$ denotes the mutant
  representation associated to the representation $\rho\colon G_{K} \to \SU$ and $\tau^{\sharp}\colon H^1_\rho(M_K) \to H^1_{\rho^\tau}(M_{K^\tau})$ is the isomorphism of Equation~(\ref{isomutant}).
\end{theorem}

The rest of this section is devoted to the proof of Theorem~\ref{theorem:invariance_tors_mutation}. The proof is divided
into two parts: in the first one we are interested in the 
``twisted part'' of the torsion, and in the second one in its ``sign part''.

We compute the torsions in the geometric bases described as follows.
Fix a presentation of the group $G_K$:
\[
\Gamma_K = \langle x_1, \ldots, x_n \; |\; r_1, \ldots, r_{n-1}\rangle.
\]
It is known, using a result due to Waldhausen~\cite{W}, that $M_K$ has the same simple homotopy type as the two-dimensional CW-complex $W_K$ constructed as follows. The $0$-skeleton of $W_K$ consists in a single $0$-cell, its $1$-skeleton is a wedge of $n$ oriented circles corresponding to the generators $x_1, \ldots, x_n$ and the $2$-skeleton consists in $(n-1)$ $2$-cells $D_1, \ldots, D_{n-1}$ where the attaching maps are given by the relations $r_1, \ldots, r_{n-1}$.  Let us write 
$$T^K_\rho(v) = \mathrm{Tor}(C_*(W_K; Ad \circ \rho); \{\varphi_\rho(v), h^{(2)}_\rho\}) = \mathrm{Tor}(C_*(M_K; Ad \circ \rho); \{\varphi_\rho(v), h^{(2)}_\rho\}),$$
 and 
 $$\varepsilon^K = \mathrm{sgn}(\mathrm{Tor}(C_*(W_K; \IR); \mathfrak{o})) = \mathrm{sgn}(\mathrm{Tor}(C_*(M_K; \IR); \mathfrak{o})).$$
  One has
\[
\tau^K_\rho(v) = \varepsilon^K \cdot T^K_\rho(v).
\]

\subsection{Computation of the twisted part of the torsion}
\label{SS:torsiontwist}

Here we compute the twisted part of the torsions of $M_K$ and $M_{K^\tau}$ using the
Mayer--Vietoris formula respectively associated to the splittings 
$M_K = M_{1} \cup_{\mathrm{id}} M_{2}$ and $M_{K^\tau} = M_{1} \cup_{\tau} M_{2}$.

\begin{remark}
In the case of a positive mutation, using Lemma~\ref{lemma:varphi_i}, observe that the meridian of $K$ and the one of $K^\tau$ can be defined by the same loop in $M_2$ (more precisely in the boundary of the mutation sphere). As a consequence, we can choose the same $(Ad \circ \rho)$-invariant vector $P_\rho \in H^0_\rho(T^2)$ for both $K$ and its positive mutant $K^\tau$. In the sequel, we will do that.
\end{remark}

\begin{proof}[Proof of 
Theorem~\ref{theorem:invariance_tors_mutation}, ``twisted part'']

	Let $v \in H^1_\rho(M_K)$ be a non zero vector.	
	The Mayer-Vietoris sequence for twisted cohomology associated to the splitting $M_K = M_1 \cup_{\mathrm{Id}} M_2$ is the following exact sequence denoted $\mathcal{H}^*$:
\begin{equation}\label{MV2}
\mathcal{H}^* : \xymatrix@1{0 \ar[r] & H^1_{\rho}(M_K) \ar[r]^-f &H^1_{\rho_1}(M_1) \oplus H^1_{\rho_2}(M_2) \ar[r]^-{i_1^* + i_2^*} & H^1_{\rho_F}(F) \ar[r]^-\bord & H^2_{\rho}(M_K) \ar[r] &0.}
\end{equation}
Applying the Mayer-Vietoris formula for the torsions gives us:
\begin{equation}\label{formule1}
T^K_\rho(v) \cdot \mathrm{Tor}(F; Ad{\rho_F}, \mathbf{h}_F) \cdot\mathrm{Tor}(\mathcal{H}^*) = (-1)^n \mathrm{Tor}(M_1; Ad \rho_1, \mathbf{h}_{M_1})\cdot\mathrm{Tor}(M_2; Ad \rho_2, \mathbf{h}_{M_2}),
\end{equation}
where $\mathbf{h}_{M_i}$ is a basis of $H^1_{\rho_i}(M_i)$, $i=1,2$,  and $\mathbf{h}_{F}$ is a basis of $H^1_{\varphi}(F)$.

	On the other hand, the Mayer-Vietoris sequence for twisted cohomology associated to the splitting $M_{K^\tau} = M_1 \cup_{\tau} M_2$ is the exact sequence denoted $\mathcal{H}^*_\tau$:
\begin{equation}\label{MV3}
\mathcal{H}^*_\tau : \xymatrix@1{0 \ar[r] & H^1_{\rho^\tau}(M_{K^\tau}) \ar[r]^-{f^\tau} &H^1_{\rho_1^\tau}(M_1) \oplus H^1_{\rho_2^\tau}(M_2) \ar[r]^-{\tau^*i_1^* + i_2^*} & H^1_{\rho_F}(F) \ar[r]^-{\bord^\tau} & H^2_{\rho^\tau}(M_{K^\tau}) \ar[r] &0.}
\end{equation}
Here, observe that $\rho^\tau_1 = {}^{x^{-1}}\!\rho_1$ et $\rho^\tau_2=\rho_2$. 
Another application of the Mayer-Vietoris formula for the torsions gives:
\begin{equation}\label{formule2}
T^{K^\tau}_{\rho^\tau}(\tau^\sharp(v)) \cdot \mathrm{Tor}(F; Ad{\rho_F}, \bar{\mathbf{h}}_F)\cdot \mathrm{Tor}(\mathcal{H}^*_\tau) = (-1)^n \mathrm{Tor}(M_1; Ad \rho^\tau_1, \bar{\mathbf{h}}_{M_1})\cdot\mathrm{Tor}(M_2; Ad  \rho_2, \bar{\mathbf{h}}_{M_2}),
\end{equation}
where $\bar{\mathbf{h}}_{M_i}$ is a basis of $H^1_{\rho_i^\tau}(M_i)$, $i=1,2$,  et $\bar{\mathbf{h}}_{F}$ is a basis of $H^1_{\varphi}(F)$. 

	Let $\mathbf{h}_{M_i}$ be any basis of $H^1_{\rho_i}(M_i)$, $i=1,2$, and $\mathbf{h}_F$ any basis of $H^1_\varphi(F)$. We choose for basis of $H^1_{\rho^\tau_1}(M_1)$ the following basis $\bar{\mathbf{h}}_{M_1}=\bar{\phi_x}(\mathbf{h}_{M_1})$, and for basis of $H^1_{\rho^\tau_2}(M_2)$ the (original) basis $\mathbf{h}_{M_2}$ of $H^1_{\rho_2}(M_2)$. Formulas~(\ref{formule1})~\&~(\ref{formule2}) give us:
$$\frac{T^{K^\tau}_{\rho^\tau}(\tau^\sharp(v))}{T^K_\rho(v)}
=\frac{\mathrm{Tor}(\mathcal{H})}{\mathrm{Tor}(\mathcal{H}_\tau)}.$$

	Now we have to compare the torsions $\mathrm{Tor}(\mathcal{H})$ and $\mathrm{Tor}(\mathcal{H}_\tau)$. 
\begin{claim}
	For a positive mutation $\tau$, one has: $$\frac{T^{K^\tau}_{\rho^\tau}(\tau^\sharp(v))}{T^K_\rho(v)}
=\frac{\mathrm{Tor}(\mathcal{H})}{\mathrm{Tor}(\mathcal{H}_\tau)} = 1.$$ 
\end{claim}
\begin{proof}[Proof of the claim]
Let us compute in parallel the two Reidemeister torsions $\mathrm{Tor}(\mathcal{H})$ and $\mathrm{Tor}(\mathcal{H}_\tau)$:
\begin{enumerate}
  \item Let $\mathbf{b}$ be a basis of $\im( i_1^* + i_2^*)$, $\widetilde{h}^{(2)}_\rho$ be a lift of $h^{(2)}_\rho$ by $\bord$ and $c = f(v)$ a generator of $\im \, f$ (see Sequence~(\ref{MV1})). The torsion $\mathrm{Tor}(\mathcal{H})$ is equal to:
\[
\mathrm{Tor}(\mathcal{H}) = [\mathbf{b}\widetilde{h}^{(2)}_\rho/h_F]\cdot [c\widetilde{\mathbf{b}}/\mathbf{h}_{M_1}\mathbf{h}_{M_2}]^{-1}.
\]
  \item In the same way, let $\mathbf{b}'$ be a basis of $\im( \tau^*i_1^* + i_2^*)$, $\widetilde{h}^{(2)}_{\rho^\tau}$ be a lift of $h^{(2)}_{\rho^\tau}$ by $\bord^\tau$ and $c' = f^\tau\circ \tau^\sharp (v)$ be a generator of $\im f^\tau$ (see Sequence~(\ref{MV3})). The torsion $\mathrm{tor}(\mathcal{H}_\tau)$ is equal to:
\[
\mathrm{tor}(\mathcal{H}_\tau) = [\mathbf{b}'\widetilde{h}^{(2)}_{\rho^\tau}/\bar{\mathbf{h}}_F] \cdot [c'\widetilde{\mathbf{b}'}/\bar{\mathbf{h}}_{M_1}\bar{\mathbf{h}}_{M_2}]^{-1}.
\]
\end{enumerate}

	Further observe that $c' = \bar{\phi_x}\oplus \mathrm{Id}(c)$, thus $[c\widetilde{\mathbf{b}}/\mathbf{h}_{M_1}\mathbf{h}_{M_2}] = [c'\widetilde{\mathbf{b}'}/\bar{\mathbf{h}}_{M_1}\bar{\mathbf{h}}_{M_2}]$. As a result we obtain:
\[
\frac{T^{K^\tau}_{\rho^\tau}(\tau^\sharp(v))}{T^K_\rho(v)} = [\mathbf{b}\widetilde{h}^{(2)}_\rho/\mathbf{b}'\widetilde{h}^{(2)}_{\rho^\tau}].
\]

It remains to compute the following bases change determinant $[\mathbf{b}\widetilde{h}^{(2)}_\rho/\mathbf{b}'\widetilde{h}^{(2)}_{\rho^\tau}]$. It is easy to observe that $$[\mathbf{b}'\widetilde{h}^{(2)}_{\rho^\tau}/\mathbf{b}\widetilde{h}^{(2)}_\rho] = [\mathbf{b}'/\mathbf{b}] \cdot [\widetilde{h}^{(2)}_{\rho^\tau} / \widetilde{h}^{(2)}_{\rho}]$$
and the computation of this bases change determinant uses the following commutative diagram (see Claim~\ref{claim:commutativity}): 
\[
\xymatrix{\cdots \ar[r] & H^1_{\rho_1}(M_1) \oplus H^1_{\rho_2}(M_2) \ar[r] \ar[d]^-{\bar{\phi_x}\oplus \mathrm{Id}} & H^1_{\rho_F}(F) \ar[r]^-{\bord} \ar[d]^-{=} & H^2_\rho(M_K) \ar[r] & 0 \\
\cdots \ar[r] & H^1_{\rho^\tau_1}(M_1) \oplus H^1_{\rho^\tau_2}(M_2) \ar[r] & H^1_{\rho_F}(F)  \ar[r]^-{\bord^\tau} & H^2_{\rho^\tau}(M_{K^\tau}) \ar[r] & 0  }
\] 
The computation is divided into two parts:
\begin{enumerate}
  \item{\emph{Computation of $[\mathbf{b}'/\mathbf{b}]$.}} One can observe that $$[\mathbf{b}'/\mathbf{b}] = \det ((\overline{\tau^*i_1^*+i_2^*}) \circ (\overline{i_1^*+i_2^*})^{-1})$$ where 
$$\overline{i_1^*+i_2^*} \colon \xymatrix@1@-.5pc{H^1_{\rho_1}(M_1) \oplus H^1_{\rho_2}(M_2) / \ker(i_1^*+i_2^*) \ar[r]^-{\cong} & \im(i_1^*+i_2^*)}$$ 
and 
$$\overline{\tau^*i_1^*+i_2^*} \colon \xymatrix@1@-.5pc{H^1_{\rho_1}(M_1) \oplus H^1_{\rho_2}(M_2) / \ker(\tau^*i_1^*+i_2^*) \ar[r]^-{\cong} & \im(\tau^*i_1^*+i_2^*)},$$ are respectively induced by $i^*_1+i^*_2$ and $\tau^*i^*_1+i^*_2$. 
The action of $\tau$ on the character variety of $\pi_1(F)$ is up to conjugation  trivial (see Lemma~\ref{lemma:tau_gievn_Ad}). As a consequence $\tau^* \colon H^1_{\rho_F}(F) \to H^1_{\rho_F}(F)$ is the identity, which gives us $[\mathbf{b}'/\mathbf{b}]=1$.

  \item{\emph{Computation of $[\widetilde{h}^{(2)}_{\rho^\tau} / \widetilde{h}^{(2)}_{\rho}]$.}} In this part, we prove that $[\widetilde{h}^{(2)}_{\rho^\tau} / \widetilde{h}^{(2)}_{\rho}]= 1$ for a positive mutation. Actually, we prove that $\bord^\tau(\widetilde{h}^{(2)}_{\rho})$ is exactly the reference generator $h^{(2)}_{\rho^\tau}$ of $H^2_{\rho^\tau}(M_{K^\tau})$. 

	Let us recall the precise definition of the reference generators $h^{(2)}_{\rho}$ and $h^{(2)}_{\rho^\tau}$. In the case of a positive mutation, the meridian $\mu$ of $K$ and the meridian $\mu^\tau$ of $K^\tau$ are represented by the same circle in the boundary of the mutation sphere $F$ (this circle of course bounds a disk in $N(K)$ or in $N(K^\tau)$).

Let $\iota : \bord M_K \hookrightarrow M_K$ and $\iota_\tau : \bord M_{K^\tau} \hookrightarrow M_{K^\tau}$ be the usual inclusions. Consider $c \in H^2(\bord M_K; \IR) = \mathrm{Hom}(H_2(\bord M_K; \ZZ), \IR)$ and $c^\tau \in H^2(\bord M_{K^\tau}; \IR) = \mathrm{Hom}(H_2(\bord M_{K^\tau}; \ZZ), \IR) $ the fundamental classes in $H_2(\bord M_K; \ZZ)$ and $H_2(\bord M_{K^\tau}; \ZZ)$ respectively. By the definition, one has:
\[
P_\rho \smile \iota^*(h^{(2)}_\rho) = c  \text{ and }  P_{\rho^\tau} \smile \iota_\tau^*(h^{(2)}_{\rho^\tau}) = c^\tau.
\]
Further observe that the orientation of $S^3$ induces the same orientation on $M_K$ and $M_{K^\tau}$ because we use the invariant part $M_2$ to define it.

	The boundary of the mutation sphere $F$ is the disjoint union of four circles denoted $S^1_1, \ldots, S^1_4$ (\cf Fig.~\ref{Fig:F}). The Mayer-Vietoris sequence  combines with the restriction homomorphism onto the boundary to give us the following commutative diagrams:
\[
\xymatrix{H^2_\rho(M_K) \ar[r]^-{\iota^*} & H^2_\rho(\bord M_K) \ar[r]^-{P_\rho \smile \cdot} & H^2(\bord M_K; \IR) \\
H^1_{\rho_F}(F) \ar[u]^-{\bord} \ar[r]^-{(\iota_1^*, \ldots, \iota_4^*)} & \displaystyle{\bigoplus_{i=1}^4 H^1_{\rho_F}(S^1_i)} \ar[u] \ar[r] & \displaystyle{\bigoplus_{i=1}^4 H^1(S^1_i; \IR)} \ar[u]_-{\delta}}
\]
and
\[
\xymatrix{
H^1_{\rho_F}(F) \ar[r]^-{(\iota_1^*, \ldots, \iota_4^*)} \ar[d]_-{\bord^\tau}& \displaystyle{\bigoplus_{i=1}^4 H^1_{\rho_F}(S^1_i)} \ar[d] \ar[r] & \displaystyle{\bigoplus_{i=1}^4 H^1(S^1_i; \IR)} \ar[d]^-{\delta^\tau} \\
H^2_{\rho^\tau}(M_{K^\tau}) \ar[r]^-{\iota_\tau^*} & H^2_{\rho^\tau}(\bord M_{K^\tau}) \ar[r]^-{P_{\rho^\tau} \smile \cdot} & H^2(\bord M_{K^\tau}; \IR)}
\]
Observe that $\delta(t) = c$ if and only if $\delta^\tau(t) = c^\tau$. Thus $\bord (\widetilde{h}^{(2)}_{\rho}) = {h}^{(2)}_{\rho}$ and $\bord^\tau(\widetilde{h}^{(2)}_{\rho}) = {h}^{(2)}_{\rho^\tau}$, as a conclusion $[\widetilde{h}^{(2)}_{\rho^\tau} / \widetilde{h}^{(2)}_{\rho}]=1$.
\end{enumerate}
\end{proof}
\end{proof}

\subsection{Computation of the sign part of the torsion}
\label{SS:torsionsign}

We are interested in \emph{the sign of the torsions} of $M_K$ and
$M_{K^\tau}$. To this purpose we compute the torsions of $C_{*}(M_K; \IR)$
and $C_{*}(M_{K^\tau}; \IR)$ by using the Mayer--Vietoris sequences with
real coefficients associated to the splittings 
$M_K = M_{1} \cup_{\mathrm{id}} M_{2}$ and 
$M_{K^\tau} = M_{1} \cup_{\tau} M_{2}$.

\begin{remark}
  From our assumption that $\tau$ is positive and
  Lemma~\ref{lemma:varphi_i} and the following
  Remark~\ref{remark:compare_homology_under_mutation}, we can choose
  the same bases for both Mayer--Vietoris sequences with real
  coefficients of $M_K = M_1 \cup_{\mathrm{id}} M_2$ and $M_{K^\tau} =
  M_1 \cup_\tau M_2$.
\end{remark}

A consequence of the preceding Propositions~\ref{prop:homology_F} \& 
\ref{prop:homology_M_i} is that the Mayer--Vietoris
sequence $\mathcal{V}_\IR$ splits into two short exact sequences:
\[
\mathcal{V}_1 = 
   0 \to 
     H_{1}(F;\IR) \xrightarrow{(i^1_{(1)}, i^2_{(1)})} 
     H_{1}(M_1;\IR) \oplus H_1(M_2;\IR) \xrightarrow{j^1_{(1)} - j^2_{(1)}}
     H_1(M_K;\IR) \to
   0
\]
and
\[
\mathcal{V}_0 = 
  0  \to 
  H_0(F;\IR) \xrightarrow{(i^1_{(0)}, i^2_{(0)})} 
  H_{0}(M_1;\IR) \oplus H_0(M_2;\IR) \xrightarrow{j^1_{(0)} - j^2_{(0)}} 
  H_0(M_K;\IR) \to
  0.
\]
The Reidemeister torsion of $\mathcal{V}_\IR$ (here with real
coefficients) is thus:
\[
\Tor{\mathcal{V}_\IR}{\hbasis{*}_{\mathcal{V}_\IR}}{\emptyset} = 
\Tor{\mathcal{V}_0}{\hbasis{*}_{\mathcal{V}_0}}{\emptyset} \cdot
\Tor{\mathcal{V}_1}{\hbasis{*}_{\mathcal{V}_1}}{\emptyset}^{-1}
\]
where $\hbasis{*}_{\mathcal{V}_\IR}$, $\hbasis{*}_{\mathcal{V}_0}$ and 
$\hbasis{*}_{\mathcal{V}_1}$ denote bases of homology groups 
in the exact sequences.

Corresponding to the splitting $M_{K^\tau} = M_{1} \cup_{\tau} M_{2}$, the
Mayer--Vietoris with real coefficients $\mathcal{V}^\tau_\IR$ splits into
short exact sequences:
\[
\mathcal{V}^\tau_1 = 
  0 \to
    H_{1}(F) \xrightarrow{(i^1_{(1)}\circ\tau_{(1)}, i^2_{(1)})}
    H_{1}(M_1) \oplus H_1(M_2) \xrightarrow{j^1_{(1)} - j^2_{(1)}} 
    H_1(M_{K^\tau}) \to
  0
\]
and
\[
\mathcal{V}^\tau_0 =
  0 \to
  H_0(F) \xrightarrow{(i^1_{(0)}\circ\tau_{(0)}, i^2_{(0)})} 
  H_{0}(M_1) \oplus H_0(M_2) \xrightarrow{j^1_{(0)} - j^2_{(0)}}
  H_0(M_{K^\tau}) \to 
  0.
\]
Moreover, we can see the generator $\mu^\tau$ of $H_1(M_{K^\tau};\IR)$ 
is the image $j^2_{(1)}(\eta_2)$
 since $\eta_2 \in H_1(M_2;\IR)$ is a loop which bounds a disk in $N(K_2)$ (see Lemma~\ref{lemma:varphi_i}). 
The Reidemeister torsion of
$\mathcal{V}^\tau_\IR$ (here with real coefficients) is thus:
\[
  \Tor{\mathcal{V}^\tau_\IR}{\hbasis{*}_{\mathcal{V}_\IR}}{\emptyset} = 
  \Tor{\mathcal{V}_0}{\hbasis{*}_{\mathcal{V}_0}}{\emptyset} \cdot
  \Tor{\mathcal{V}_1}{\hbasis{*}_{\mathcal{V}_1}}{\emptyset}^{-1}
\]
Here we omit the bases of homology groups in 
the Mayer--Vietoris sequence~$\mathcal{V}$ in the torsions for simplicity.

We obtain:
\[
  \frac{
    \Tor{C_*(M_{K^\tau};\IR)}{\cbasis{*}_\IR}{\hbasis{*}_\IR}
  }{
    \Tor{C_*(M_K;\IR)}{\cbasis{*}_\IR}{\hbasis{*}_\IR}
  } =
  \frac{
    \Tor{\mathcal{V}^\tau_\IR}{\hbasis{*}_{\mathcal{V}_\IR}}{\emptyset}
  }{
    \Tor{\mathcal{V}_\IR}{\hbasis{*}_{\mathcal{V}_\IR}}{\emptyset}  } =
  \det \tau_{(1)} \cdot 
  \det(\bbrack{\mu}/\bbrack{\mu^\tau}) \cdot 
  (\det \tau_{(0)})^{-1}.
\]
where $C_*(M_K;\IR)$ and $C_*(M_{K^\tau};\IR)$ are endowed with the homology
orientation $\{\bbrack{pt}, \bbrack{\mu}\}$ and
$\{\bbrack{pt}, \bbrack{\mu^\tau}\}$.

Next we compute each terms in the right hand side of the preceding equality.

\begin{claim}
  The sign of $\det \tau_{(0)}$ is positive.
\end{claim}
\begin{proof}
  As $\tau_{(0)}$ is just the identity map, it maps the
  class of the point to itself.
\end{proof}

\begin{claim}
  One has $\det(\bbrack{\mu}/\bbrack{\mu^\tau}) = +1$.
\end{claim}
\begin{proof}
  This equality comes from our convention. Both meridians $\mu$ and
  $\mu^\tau$ are given by the same loop in $M_2$ and does not have any
  influence from a mutation.
\end{proof}

\begin{claim}\label{claim:sgnTR}
  The sign of $\det \tau_{(1)}$ is positive.
\end{claim}
\begin{proof}
  The homology group $H_1(F; \ZZ)$ is isomorphic to the following
  quotient
  \[
  H_1(F) \simeq 
    \IR a \oplus \IR b \oplus \IR c \oplus \IR d /
    (a+b+c+d = 0).
  \]
  The isomorphism of $H_1(F)$ induced by a positive mutation $\tau$ is
  one of the following:
  \begin{align*}
    \tau_{(1)}&\colon a \mapsto b, \; b \mapsto a,\; c \mapsto d,\; d \mapsto c,\\
    \tau_{(1)}&\colon a \mapsto c, \; b \mapsto d,\; c \mapsto a,\; d \mapsto b,\\
    \tau_{(1)}&\colon a \mapsto d, \; b \mapsto c,\; c \mapsto b,\; d
    \mapsto a.
  \end{align*}
  In each case it is easy to see that $\det \tau_{(1)}$ is $+1$.
\end{proof}

Using the three previous claims, we conclude that:
\begin{claim}
  We have:
  \[
  \mathrm{sgn}(\Tor{\mathcal{V}^\tau}{\hbasis{*}_{\mathcal{V}_\IR}}{\emptyset}) =
  \mathrm{sgn}(\Tor{\mathcal{V}}{\hbasis{*}_{\mathcal{V}_\IR}}{\emptyset}).
  \]
\end{claim}

\section{Conclusion and open questions}
\label{S:open}

Among the Reidemeister torsion form, another important invariant in Reidemeister torsions theory is the so--called \emph{twisted Alexander invariant}, which can be understand as a non abelian version of the well--known Alexander polynomial. J. Milnor~\cite{Milnor:1962} proved that the (usual) Alexander polynomial ${\Delta_K(t)}$ can be interpreted as an abelian Reidemeister torsion: let $K \subset S^3$ be a knot and consider the abelianization $\alpha\colon G_K \to \ZZ$ defined by $\alpha(\mu) = t$ which extends into $\alpha\colon \ZZ[G_K] \to \QQ(t)$. Milnor proves that the twisted complex $C_*(M_K; \QQ(t)) = \QQ(t) \otimes C_*(\widetilde{M_K}; \ZZ)$ is acyclic and its torsion is expressed using the Alexander polynomial (up to $\pm t^{\pm m}$, $m \in \NN$):
\[
\mathrm{tors}(M_K; \alpha) = \frac{\Delta_K(t)}{t-1}.
\]
The twisted Alexander invariant of $K$, twisted by a generic representation $\rho \colon G_K \to \SL$, is:
\[
\Delta_K^{Ad\circ \rho}(t) = \mathrm{tors}(C_*(M_K; Ad\circ \rho \otimes \alpha)).
\]

In \cite{YY:Fourier}, Y. Yamaguchi proved a formula which make a link between the twisted Alexander invariant $\Delta_K^{Ad\circ \rho}(t)$ -- an acyclic torsion -- with a special value of the torsion form $\tau^K_\rho$, if $\rho \colon G_K \to \SL$.

Numerical computations made by N. Dunfield, S. Friedl and N. Jackson~\cite{DFJ} show that the twisted Alexander invariant is not invariant by mutation, but the torsion form is. An interesting question will be to understand and to characterize the ``default'' between $\Delta_K^{Ad\circ \rho}(t)$ and $\Delta_{K^\tau}^{Ad\circ \rho^\tau}(t)$ and especially at the discrete and faithful representation corresponding to the complete structure for hyperbolic knots.

\bibliographystyle{QT}
\bibliography{ref}

\end{document}